\documentclass[notitlepage]{scrartcl}
\usepackage[shortlabels]{enumitem}
\usepackage{amsmath}
\usepackage{amssymb}
\usepackage{amsthm}
\usepackage{abstract}
\usepackage{xcolor}
\usepackage{comment}
\usepackage{mathtools}
\usepackage{graphicx}
\usepackage{caption}
\usepackage{subcaption}
\usepackage{float}
\usepackage{hyperref}
\hypersetup{
    colorlinks=true,
    linkcolor=blue,
    filecolor=magenta,      
    urlcolor=cyan
    }
\makeatletter
\newcommand*{\barfix}[2][.175ex]{%
  \mathpalette{\@barfix{#1}}{#2}%
}
\newcommand*{\@barfix}[3]{%
  \vbox{%
    \kern#1\relax
    \hbox{$#2#3\m@th$}%
  }%
}
\makeatother

\newtheorem{theorem}{Theorem}
\newtheorem{thm}{Theorem}[section]

\newtheorem{lemma}[thm]{Lemma}
\newtheorem{proposition}[thm]{Proposition}

\newtheorem{remark}[thm]{Remark}

\newtheorem{question}[thm]{Question}
\newcommand{\footremember}[2]{%
    \footnote{#2}
    \newcounter{#1}
    \setcounter{#1}{\value{footnote}}%
}

\usepackage[margin=1in]{geometry}

\title{\vspace{-1.5cm}Climbing up a random subgraph of the hypercube} 
\author{%
Michael Anastos \footremember{alley}{\scriptsize{ Institute of Science and Technology Austria (ISTA), Klosterneurburg 3400, Austria. Email:  michael.anastos@ist.ac.at. Research supported by the European Union’s Horizon 2020 research and innovation programme under the Marie Sk\l{}odowska-Curie grant agreement No.\ 101034413 \includegraphics[width=4.5mm, height=3mm]{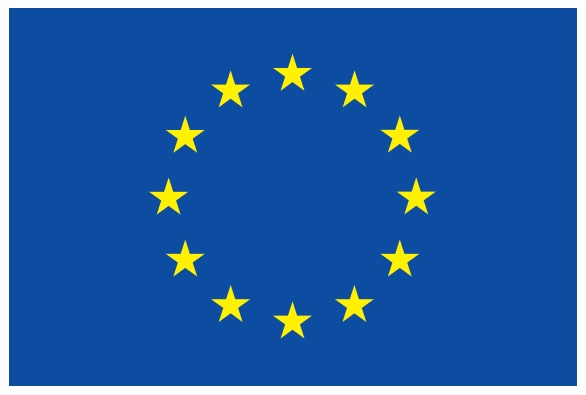} }}%
\and Sahar Diskin \footremember{alley2}{\scriptsize{School of Mathematical Sciences, Tel Aviv University, Tel Aviv 6997801, Israel. Email: sahardiskin@mail.tau.ac.il.}}%
\and Dor Elboim \footremember{trailer}{\scriptsize{School of Mathematics, Institute for Advanced Study, Princeton, N.J. 08540, U.S.A.. Email: dorelboim@gmail.com.}}%
\and Michael Krivelevich \footremember{trailer2}{\scriptsize{School of Mathematical Sciences, Tel Aviv University, Tel Aviv 6997801, Israel. Email: krivelev@tauex.tau.ac.il.}}%
}
\begin{document}
\maketitle
\vspace{-2em}
\begin{abstract}
    Let $Q^d$ be the $d$-dimensional binary hypercube. We say that $P=\{v_1,\ldots, v_k\}$ is an increasing path of length $k-1$ in $Q^d$, if for every $i\in [k-1]$ the edge $v_iv_{i+1}$ is obtained by switching some zero coordinate in $v_i$ to a one coordinate in $v_{i+1}$. 
    
    Form a random subgraph $Q^d_p$ by retaining each edge in $E(Q^d)$ independently with probability $p$. We show that there is a phase transition with respect to the length of a longest increasing path around $p=\frac{e}{d}$. Let $\alpha$ be a constant and let $p=\frac{\alpha}{d}$. When $\alpha<e$, then there exists a $\delta \in [0,1)$ such that \textbf{whp} a longest increasing path in $Q^d_p$ is of length at most $\delta d$. On the other hand, when $\alpha>e$, \textbf{whp} there is a path of length $d-2$ in $Q^d_p$, and in fact, whether it is of length $d-2, d-1$, or $d$ depends on whether the all-zero and all-one vertices percolate or not.
\end{abstract}

\section{Introduction}
\subsection{Background and main result}
The $d$\textit{-dimensional hypercube}, $Q^d$, is the graph whose vertex set is $V(Q^d)\coloneqq \{0,1\}^d$, and where two vertices are adjacent if they differ in exactly one coordinate. The hypercube and its subgraphs rise naturally in many contexts and have received much attention in combinatorics, probability, and computer science. 

The \textit{random subgraph }$Q^d_p$ is obtained by retaining each edge in $E(Q^d)$ independently with probability $p$. Several phase transitions have been observed in $Q^d_p$. Indeed, the study of $Q^d_p$ was initiated by Sapo\v{z}enko \cite{S67} and Burtin \cite{B67}, who showed that the sharp threshold for connectivity is $\frac{1}{2}$: when $p<\frac{1}{2}$, \textbf{whp}\footnote{With high probability, that is, with probability tending to one as $d$ tends to infinity.} $Q^d_p$ is disconnected, whereas for $p>\frac{1}{2}$, \textbf{whp} $Q^d_p$ is connected. This result was subsequently strengthened by Erd\H{o}s and Spencer \cite{ES79} and by Bollob\'as \cite{B83}. Bollob\'as further showed \cite{B90} that $p=\frac{1}{2}$ is the threshold for the existence of a perfect matching. Recently, resolving a long-standing open problem, Condon, Espuny D\'iaz, Gir\~ao, K\"uhn, and Osthus \cite{CDGKO21} showed that $p=\frac{1}{2}$ is also the threshold for the existence of a Hamilton cycle in $Q^d_p$.

In the sparser regime, Erd\H{o}s and Spencer conjectured that $Q^d_p$ undergoes a phase transition with respect to its component structure around $p=\frac{1}{d}$, similar to that of $G(n,p)$ around $p=\frac{1}{n}$. This conjecture was confirmed by Ajtai, Koml\'os, and Szemer\'edi \cite{AKS81}, with subsequent work by Bollob\'as, Kohayakawa, and \L{}uczak \cite{BKL92}. Given $\alpha>1$, let $\zeta_{\alpha}$ be the unique solution in $(0,1)$ of the equation
\begin{align}\label{survival prob}
    \zeta_{\alpha}=1-\exp(-\alpha\zeta_{\alpha}).
\end{align}
Then, Ajtai, Koml\'os, and Szemer\'edi \cite{AKS81}, and Bollob\'as, Kohayakawa, and \L{}uczak \cite{BKL92} showed that when $p=\frac{1-\epsilon}{d}$, for $\epsilon>0$, \textbf{whp} all components of $Q^d_p$ have order $O_{\epsilon}(d)$, and when $p=\frac{1+\epsilon}{d}$ \textbf{whp} $Q^d_p$ contains a unique giant component of asymptotic order $\zeta_{1+\epsilon}2^d$, and all other components have order $O_{\epsilon}(d)$.\footnote{We use the convention $f(x)=O_{r}(g(x))$ to say that there exist a constant $C$, which may depend on $r$, such that for all sufficiently large $x$ we have $f(x)\le C|g(x)|$.} We note that $\zeta_{\alpha}$ is equal to the survival probability of a Galton-Watson tree with offspring distribution Po$(\alpha)$.

In this paper, we show a new phase transition in the hypercube, which occurs when $p=\frac{e}{d}$. Before stating our result, let us introduce some notation. For a vertex $v\in Q^d$, we denote by $I(v)\subseteq [d]$ the set of coordinates of $v$ which are $1$. For $v_1, v_2\in Q^d$, we say that $v_1<v_2$ if $I(v_1)\subsetneq I(v_2)$. Given a path $P=\{v_1, \ldots, v_k\}$ in $Q^d$, we say that $P$ is an \textit{increasing path} of length $k-1$ in the hypercube, if for every $i\in [k-1]$, $|I(v_{i+1})|-|I(v_i)|=1$. Note that the longest increasing path in $Q^d$ is of length $d$. Given a subgraph $H\subseteq Q^d$, let $\ell(H)$ be the length of a longest increasing path in $H$. Our result is as follows.
\begin{theorem}\label{th: main}
Let $\alpha$ be a constant, and let $p=\frac{\alpha}{d}$. Then, the following holds.
\begin{enumerate}[(a)]
    \item For every $\alpha<e$, there exists $\delta\coloneqq \delta(\alpha)$, $0\le \delta<1$, such that \textbf{whp} $\ell(Q^d_p)\le \delta d$. \label{subcritical}
    \item For every $\alpha>e$, \textbf{whp} $\ell(Q^d_p)\ge d-2$. Furthermore, \label{supercritical}
    \begin{align*}
        &\mathbb{P}\big( \ell(Q^d_p)=d\big) =(1+o_d(1))\zeta_{\alpha}^2,\\
        &\mathbb{P}\big( \ell(Q^d_p)=d-1\big) =(1+o_d(1))2\zeta_{\alpha}(1-\zeta_{\alpha}),
    \end{align*}
    where $\zeta_{\alpha}$ is defined according to \eqref{survival prob}.
\end{enumerate}
\end{theorem}
In a sense, the above shows that for $\alpha>e$ \textbf{whp} the path is of length at least $d-2$, and whether it is of length $d-2, d-1$, or $d$ depends on whether the all-$0$-vertex and the all-$1$-vertex `percolate', that is, if the branching process rooted at the vertices survives for long enough (with respect to $d$). 

Let us mention a related result of Pinsky \cite{R13}, who considered (among other things) the \textit{number} of increasing paths of length $d$ in $Q^d_p$, where $p=\frac{\alpha}{d}$. In \cite{R13}, he showed that if $\alpha<e$, then \textbf{whp} the number of such paths is $0$. Theorem \ref{th: main}\ref{subcritical} shows that, in fact, when $\alpha<e$, typically a longest path is smaller by a multiplicative constant. Pinsky further showed that if $\alpha>e$, then the probability there are any such paths is bounded away from $0$ and $1$. Theorem \ref{th: main}\ref{supercritical} gives a detailed description of the typical length of a longest path when $\alpha>e$ (see also Remark \ref{r: distribution}).

We finish this section by noting that while some of our lemmas extend to the case where $\alpha=e$, our overall proof does not. It would be interesting to see what the behaviour is at this critical point.
\begin{question}
Let $p=\frac{e}{d}$. Form $Q^d_p$ by retaining each edge of $Q^d$ independently with probability $p$. What can be said about a longest increasing path in $Q^d_p$?
\end{question}

The structure of the paper is as follows. In the subsequent subsection, Section \ref{s: outline}, we give an outline of the proof of our main theorem. In Section \ref{s: notation} we collect some notation and lemmas which will be of use throughout the proof. Section \ref{s: subcritical} is devoted to the proof of Theorem \ref{th: main}\ref{subcritical}, and Section \ref{s: supercritical} is devoted to the proof of Theorem \ref{th: main}\ref{supercritical}.

\subsection{Proof outline}\label{s: outline}
The proof of Theorem \ref{th: main}\ref{subcritical} follows from a first moment argument. The proof of Theorem \ref{th: main}\ref{supercritical}, on the other hand, is far more delicate.

For proving Theorem \ref{th: main}\ref{supercritical}, we begin by showing, through a careful second-moment argument, that it is not \textit{very} unlikely to have a path between the all-$0$-vertex and the all-$1$-vertex in $Q^d_p$ --- in fact, we give a lower bound for the probability of that event which is inverse-polynomial in $d$. Then, our goal is to show that if the all-$0$-vertex and the all-$1$-vertex `percolate', then  \textbf{whp} we can find polynomially many vertex disjoint subcubes, so that the $0$ antipodal point of each of the subcubes is connected by a path in $Q^d_p$ to the all-$0$-vertex in $Q^d$, and the $1$ antipodal point of each of the subcubes is connected by a path in $Q^d_p$ to the all-$1$-vertex in $Q^d$. Since these subcubes are vertex disjoint, these events are independent for different subcubes. Thus, since the probability that in these subcubes there is an increasing path between the $0$ antipodal point and the $1$ antipodal point in $Q^d_p$ is at least inverse polynomial in $d$, having polynomially many such subcubes we will be able to conclude that \textbf{whp} in at least one of these subcubes there is an increasing path from the $0$ antipodal point to the $1$ antipodal point, which then extends to an increasing path between the all-$0$-vertex and the all-$1$-vertex in $Q^d_p$.

Finding these subcubes is the most involved part of the paper, and therein lie several novel ideas. We note that key parts of this lie in the Tree Construction algorithm and its properties, given in Section \ref{s: bfs}. 

Roughly, we show that if the all-$0$-vertex `percolates', then \textbf{whp} we can construct a `good' tree in $Q^d_p$, denote it by $T_0$, which is rooted at the all-$0$-vertex. We note that $T_0$ is monotone increasing, that is, the layers of the tree correspond to the layers of the hypercube. We construct this tree so that it has polynomially in $d$ many leaves, all of which reside in a relatively low layer (that is, $O(\log d)$). Furthermore, we assign a set of coordinates $C_v$ to every leaf $v\in V(T_0)$, where $C_v\cap I(v)=\varnothing$. While constructing this tree, we ensure that this list is of size at most polylogarithmic in $d$, and that for every other leaf $v'\in V(T_0)$, there is some $i\in C_v$ such that $i\notin I(v)$ and $i\in I(v')$ (we informally say that the leaves of the tree are `easily distinguishable'). Let us denote by $V_0$ the set of leaves of $T_0$ (see Figure \ref{f: first outline}).

Then, we show that if the all-$1$-vertex `percolates', then \textbf{whp} we can construct a `good' tree in $Q^d_p$, denote it by $T_1$, which is rooted at the all $1$-vertex. Here $T_1$ is monotone decreasing, that is, the layers of the tree correspond (in decreasing order) to the layers of the hypercube. The set of leaves of this tree, $V_1$, resides in a relatively high layer (that is, $d-O(\log d)$), and we have that $|V_1|=d|V_0|$. Moreover, we will construct the tree so that all of its vertices are above the vertices of $T_0$, that is, for every $v_0\in V_0$ and $v_1\in V_1$, we have that $I(v_0)\subset I(v_1)$. While this might seem a stringent requirement, note that given $p=\frac{\alpha}{d}$ with $\alpha>e$, we have that $p\cdot\frac{d}{2}>1$, and thus we can restrict the growth of the trees to half of the coordinates while remaining supercritical. Indeed, at a certain point through the construction of the trees, we will grow $T_0$ only on the first $\frac{d}{2}$ coordinates, and $T_1$ only on the last $\frac{d}{2}$ coordinates. After constructing $T_1$, we will arbitrarily associate disjoint sets of $d$ vertices in $V_1$ to every vertex in $V_0$ (see Figure \ref{f: first outline}).

\begin{figure}[H]
\centering
\includegraphics[width=0.65\textwidth]{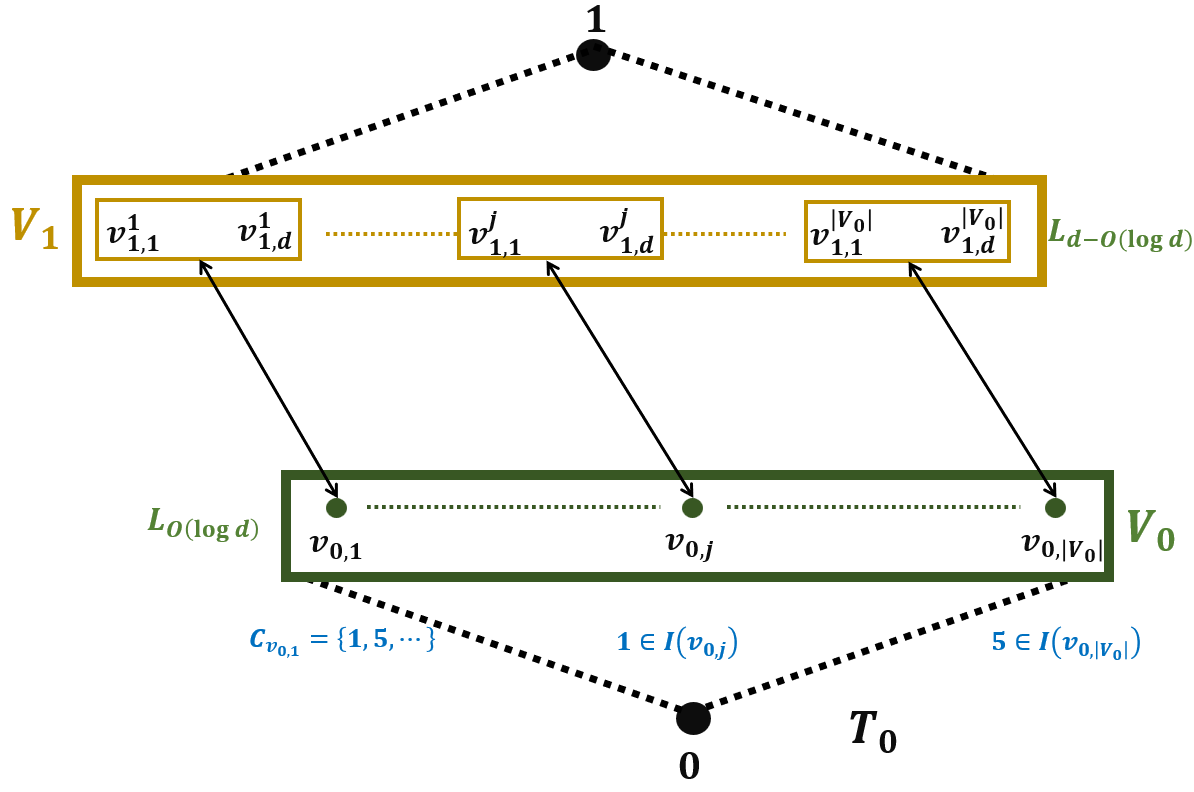}
\caption{Illustration of $T_0$, rooted at the all-$0$-vertex, and $T_1$, rooted at the all-$1$-vertex. The list of coordinates $C_{v_{0,1}}$ associated with $v_{0,1}$ is displayed, together with specific coordinates for the $j$-th and last vertex in $V_0$ which are in $C_{v_{0,1}}$. We stress that $C_{v_{0,1}}\cap I(v_{0,1})=\varnothing$. The disjoint subsets of $d$ vertices inside $V_1$ which are associated with every vertex in $V_0$ are displayed, with an arrow connecting them according to their association.}
\label{f: first outline}
\end{figure}

Since the leaves of $T_0$ are easily distinguishable, we utilise a projection-type argument to argue that given a set of, say, $d$ leaves in $T_1$ and a vertex $v_0\in V(T_0)$, \textbf{whp} we can grow in $Q^d_p$ a tree of small height, rooted at one of these leaves, such that it has a vertex $v_1$ which is above $v_0$, but is not above any other leaf of $T_0$. We do so by growing a tree, rooted at one of the $d$ leaves which were assigned to $v_0$ (which is fixed), so that all its leaves are above $v_0$, and to a sufficient height such that for every $i\in C_{v_0}$, we have that $i\in v_1$. This, in turn, will allow us to find pairs of vertices, which will form the antipodal points of the vertex disjoint subcubes we seek to construct.
\section{Preliminaries}\label{s: notation}
We begin with some notation and terminology which will be of use for us throughout the paper. Given a graph $G$, a subset $S\subseteq V(G)$, and a vertex $v\in V(G)$, we let $N_S(v)\coloneqq N(v)\cap S$, that is, the neighbourhood of $v$ in $S$. With a slight abuse of notation, given a subgraph $H\subseteq G$ and $v\in V(G)$, we define $N_H(v)\coloneqq N_{V(H)}(v)$. 

Recall that for every $v\in Q^d$, we denote by $I(v)\subseteq [d]$ the set of coordinates which are $1$. Furthermore, for $v_1, v_2\in Q^d$, we say that $v_1\le v_2$ if $I(v_1)\subseteq I(v_2)$, and $v_1<v_2$ if $I(v_1)\subsetneq I(v_2)$. Given a path $P=\{v_1, \ldots, v_k\}$ in $Q^d$, we say that $P$ is an \textit{increasing path} in the hypercube, if for every $i\in [k-1]$, $v_i<v_{i+1}$. Also, given a set of vertices $S\subseteq V(Q^d)$, let $I(S)\coloneqq \cup_{v\in S}I(v)$. 

Given two vertices $u, v\in Q^d$ such that $uv\in E(Q^d)$, we denote by $c(u,v)$ the coordinate in $[d]$ in which these two vertices differ. Given a tree $T$ rooted in $r\in V(Q^d)$ and a vertex $v\in V(T)$, let
\begin{align*}
    C_T(v)\coloneqq\big\{c(x,w)\colon x\neq v, xw\in E(T), x\text{ is on the path } P \text{ from } r \text{ to } v \text{ in } T, w\notin P \big\}.
\end{align*}
\begin{figure}[H]
\centering
\includegraphics[width=0.85\textwidth]{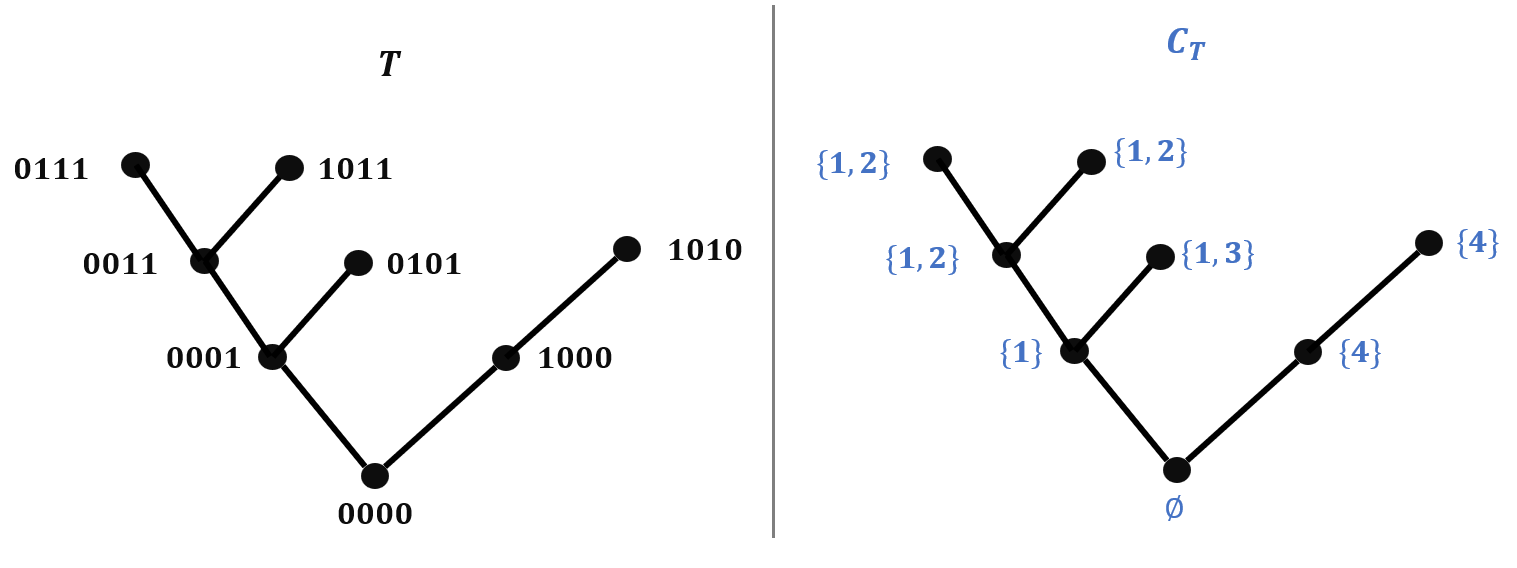}
\caption{Illustration of $C_T(v)$ for a fixed tree $T$ in $Q^4$, rooted at the all-$0$-vertex. On the left side, the tree $T$ and its vertices are presented. On the right side, the values of $C_T(v)$ for every vertex of $T$ appear in blue.}
\label{f: sat_C_0}
\end{figure}

We stress that all the trees throughout the paper are monotone, that is, given a tree $T$ rooted at $r\in V(Q^d)$, then either $r<v_1<\ldots<v_k$ for every path $\{r, v_1, \ldots, v_k\}$ of $T$, or $r>v_1>\ldots>v_k$ for every path $\{r,v_1,\ldots,v_k\}$ of $T$.

We will further make use of $C_T(v,0)\coloneqq C_T(v)\cap \left([d]\setminus I(v)\right)$ and $C_T(v,1)\coloneqq C_T(v)\cap I(v)$.

Given $v_1, v_2\in V(Q^d)$ such that $v_1\le v_2$, let $Q(v_1, v_2)$ be the induced subcube of $Q^d$ whose vertex set is given by
\begin{align*}
    V(Q(v_1, v_2))\coloneqq \left\{u\colon v_1\le u\le v_2 \right\}.
\end{align*}
Note that $Q(v_1, v_2)$ is isomorphic to a hypercube of dimension $|I(v_2)|-|I(v_1)|$. Furthermore, note that given $u_1, u_2$ and $v_1, v_2$, we have that $Q(u_1, u_2)$ and $Q(v_1, v_2)$ are guaranteed to be vertex disjoint if $I(u_1)\not\subseteq I(v_2)$, equivalently if $I(u_1)\cap ([d]\setminus I(v_2))\neq \varnothing$. 

We denote by $\textbf{0}$ the all-$0$-vertex in $Q^d$, and by $\textbf{1}$ the all-$1$-vertex in $Q^d$. For $i\in [0,d]$, we denote by $L_i$ the $i$-th layer in the hypercube, that is, the set of vertices with exactly $i$ ones.

Throughout the paper, we omit rounding signs for the sake of clarity of presentation. All the logarithms are assumed to be the natural logarithm.

We will use the following Chernoff-type bounds on the tail probabilities of the binomial distribution (see, for example, \cite[Appendix A]{AS16}).
\begin{lemma}\label{l: chernoff}
Let $N\in \mathbb{N}$, let $p\in [0,1]$ and let $X\sim Bin(N,p)$.
\begin{itemize}
    \item For every positive $t$, $$\mathbb{P}(X>tNp)\le \left(\frac{e}{t}\right)^{tNp}.$$
    \item For every $0<b\le \frac{Np}{2}$, $$\mathbb{P}(X<Np-b)\le \exp\left(-\frac{b^2}{4Np}\right).$$
\end{itemize}
\end{lemma}

\subsection{Exploring a tree with `easily distinguishable' leaves}\label{s: bfs}
We will utilise a heavily modified variant of the Breadth First Search (BFS) algorithm, which we will refer to as the \textit{Tree Construction algorithm}.
The Tree Construction algorithm is fed with the following as input.
\begin{enumerate}[label=(In\arabic*)]
    \item A subcube $H\coloneqq Q(u_0, u_1)\subseteq Q^d$, with $u_0\le u_1$, together with an order $\sigma$ on $V(H)$; and, \label{input: subcube}
    \item a vertex $r\in \{u_0, u_1\}$, which will serve as the root; and, \label{input: root} 
    \item a set of coordinates $C\subseteq [d]$, which are to be avoided; and, \label{input: coordinates}
    \item a layer $L_i\subseteq V(Q^d)$, with $|I(u_0)|\le i \le |I(u_1)|$, at which this algorithm is truncated; and, \label{input: layer}
    \item a sequence of independent Bernoulli$(p)$ random variables, $\{X_e\}_{e\in E(H)}$. \label{input: probability}
\end{enumerate}
The algorithm outputs a BFS-type tree $T$, rooted at $v$ in $H$. To that end, the algorithm maintains three sets of vertices: $B$, the set of vertices whose exploration is complete (and are part of the BFS tree); $A$, the active vertices currently being explored, kept in a \textit{queue}; and $Y$, the vertices that have not been explored yet. The algorithm starts with $T$ being the root $r$, $B=\varnothing$, $A=\{r\}$ and $Y=V(H)\setminus \{r\}$.  As long as $A$ is not empty, the algorithm proceeds as follows.

Let $x$ be the first vertex in $A$ and let $j(x)\in [0,d]$ be such that $x$ belongs to the layer $L_{j(x)}$. Let 
\begin{align*}
    C_x\coloneqq \begin{cases}
        C_T(x,0), &\quad \text{if }v=u_0\\
        C_T(x,1), &\quad \text{if }v=u_1
    \end{cases},
\end{align*}
and 
\begin{align*}
    j_x\coloneqq \begin{cases}
       j(x)+1, &\quad \text{if }v=u_0\\
       j(x)-1, &\quad \text{if }v=u_1
    \end{cases}.
\end{align*}

If $L_{j_x}=L_i$ then let $Y_x=\emptyset$. Otherwise, let

$$Y_x\coloneqq \left\{y\in Y\cap L_{j_x}: xy\in E(H) \quad \& \quad c(x,y)\in [d]\setminus \left(C\cup C_x\right) \right\},$$
that is, the set of neighbours of $x$ in $Y\cap L_{j_x}$ that differ from $x$ in a coordinate in $[d]\setminus \left(C\cup C_x\right)$. Then, for $y\in Y_x$,  we query the edge $yx$. For this, we reveal the random variable $X_{xy}$. If $X_{xy}=1$ then the edge $xy$ belongs to $H_p$ (recall that this occurs for each edge $e\in E(H)$ with probability $p$, independently), otherwise it does not. In the case $X_{xy}=1$, if so far we have identified fewer than $\log d$ edges $xy'$ with $y'\in Y_x$, that is, we have had fewer than $\log d$ random variables $X_{xy'}=1$, then we move $y$ from $Y$ to the end of $A$. If we have identified $\log d$ such edges (or if $X_{xy}=0$), we move to the next edge. Once all the edges $xy$ for every $y\in Y_x$ have been queried, we move $x$ from $A$ to $B$. The algorithm terminates once $A$ is empty. It then outputs the tree $T$, rooted at $r$, and spanned by the edges the algorithm has detected. 
 
We will utilise the following properties of the Tree Construction algorithm.
\begin{proposition}\label{proposition: tree construction}
The following properties of the Tree Construction algorithm hold.
\begin{enumerate}[(a)]
    \item If we start with $u_0$, that is $r=u_0$, then for every $u\in V(T)$ we have that $|C_T(u,0)|\le i\cdot \log  d$. Similarly, if we start with $u_1$, that is $r=u_1$, then for every $u\in V(T)$ we have that $|C_T(u,1)|\le (d-i)\log d$. \label{property: list}
    \item If we start with $u_0$, then for every two leaves $w_1,w_2\in V(T)$ we have that $I(w_2)\cap C_T(w_1,0)\neq \varnothing$. Similarly, if we start with $u_1$, then for every two leaves $w_1,w_2\in V(T)$ we have that $([d]\setminus I(w_2))\cap C_T(w_1,1)\neq\varnothing$. \label{property: leaves} 
    \item Suppose that we are at the first moment where $x$ is the first vertex in $A$ and let $A_x,B_x$ be the sets $A,B$ at that moment. Let $y\in N_H(x)\cap L_{j_x}$. If $y\in A_x\cup B_x$ then $c(x,y)\in  C_x$. \label{property: structural observation}
    \item Suppose that $H$ has dimension $(1-o(1))d$, $|C|\le \frac{d}{2}+\sqrt{d}$, the layer $L_i$ is at distance $\ell$ from the root, where $\ell=o(\log^5d), \ell=\omega(1)$, and that $p=\frac{\alpha}{d}$ for some constant $\alpha>e$. Then, with probability at least $(1+o(1))\zeta_{p(d-|C|)}$, where $\zeta_{c}$ is defined as in \eqref{survival prob}, the following holds. The Tree Construction algorithm outputs a tree $T$, such that the number of leaves in the layer $L_{i-1}$ if we start from $u_0$, and in the layer $L_{i+1}$ if we start from $u_1$, is between $\left(p(d-|C|)\right)^{0.9\ell}$ and $\left(pd\right)^{1.1\ell}$. \label{property: queries}
\end{enumerate} 
\end{proposition}
\begin{proof}
We prove each item separately.
\begin{enumerate}[(a)]
    \item The claim follows as we truncate at layer $L_i$, and the number of direct descendants of every $u\in V(T)$ in $T$, allowed by the Tree Construction algorithm, is at most $\log d$. 
    \item    If we start with $u_0$, then there exists some common ancestor $w\in V(T)$ of $w_1$ and $w_2$, and vertices $w_1^-,w_2^-$, such that $w<w_1^-\leq w_1$, $w<w_2^-\leq w_2$ and $ww_1^-,ww_2^- \in E(T)$. Let $\ell=c(w,w_2^-)$. Then $\ell\in I(w_2^-)\subseteq I(w_2)$. On the other hand, by definition $\ell\in C_T(w_1,0)$. Thus $I(w_2)\cap C_{T}(w_1,0)\neq \varnothing$. The statement for when starting with $u_1$ follows by symmetry.
    \item Let us consider the case that $r=u_0$, noting that the case $r=u_1$ follows from symmetric arguments. Let $y \in N_H(x)\cap L_{j_x}$ and suppose that $y\in A_x\cup B_x$ at that moment. Let $w$ be the unique vertex on the $x-y$ path spanned by $T$ which satisfies $w<x$ and $w<y$. Furthermore, let $x^-,y^-$ be such that $w<x^-\le x$, $w<y^-\le y$, and $wx^-, wy^-\in E(T)$ (see Figure \ref{f: tree_property}).
    Then $c(w,y^-) \in I(y^-)\subseteq I(y)$. In addition, by definition $c(w,y^-)\in C_{x^-}$. Since for any $s\in A_x$ we do not allow traversing from $s$ in $T$ along the coordinates of $C_s$, we have that $c(w,y^-)\in C_s$ for every descendant $s$ of $x^-$ in $T$, and hence $c(w,y^-)$ does not belong to $I(s)$ for any such $s$. In particular $c(w,y^-)\notin I(x)$. The above implies that $c(x,y)=c(w,y^-)\in C_x$ which completes the proof.
    \begin{figure}[H]
    \centering
    \includegraphics[width=0.4\textwidth]{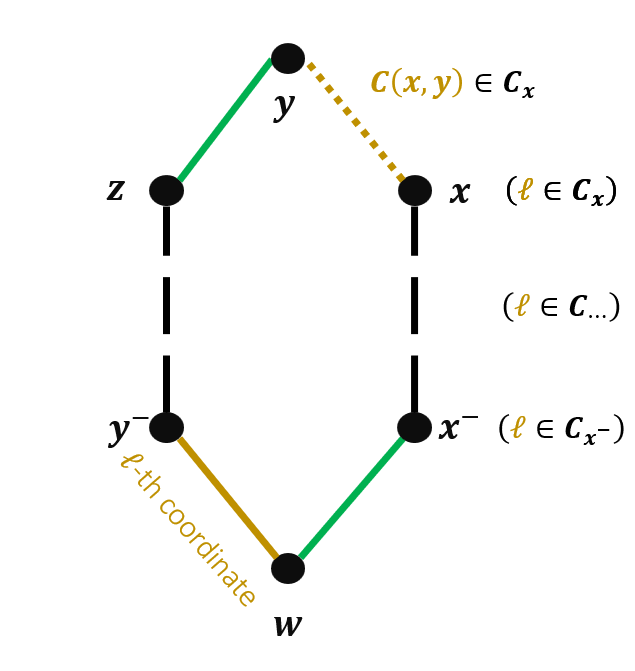}
    \caption{Illustration of the proof of Proposition \ref{proposition: tree construction}\ref{property: structural observation}. Here $c(w,y^{-})=\ell$. As for every $s\in A$, we do not traverse from $s$ on the coordinates of $C_s$ in $T$, we have that $\ell$ is in $C_s$ for every descendant $s$ of $x^-$. As both $zy$ and $xy$ are in $E(H)$, along any path from $x$ to $y$ the $\ell$-th coordinate must be traversed, and therefore $c(x,y)=\ell\in C_x$. Note that a similar argument can be made for $z$, and therefore $zy$ cannot, in fact, be in $T$. }
    \label{f: tree_property}
    \end{figure}
    \item For every $k\in [0,d]$ and every $x\in L_k$ one has that $|N_{L_{k+1}}(x)|=d-k$ and $|N_{L_{k-1}}(x)|=k$. Thus, suppose this is the first moment where $x$ is the first vertex in $A$, let $d_H$ be the dimension of $H$, and let
    \begin{align*}
        d'=\begin{cases} d_H-i-|C|, &\quad \text{if } v=u_0 \text{ and } x\notin L_{i-1},\\
        d_H-(d_H-i)-|C|, &\quad \text{if } v=u_1 \text{ and } x\notin L_{i+1},\\
        0 &\quad \text{otherwise.}
        \end{cases}
    \end{align*}
    Then, by Proposition \ref{proposition: tree construction}\ref{property: structural observation}, we have that $|Y_x|\ge d'-|C_x|$. Let $Z_x$ be the number of direct descendants of $x$ in the above process. Then, $Z_x$ stochastically dominates $\min\left\{\log d, Bin(d'-|C_x|,p)\right\}$. By Proposition \ref{proposition: tree construction}\ref{property: list}, since $L_i$ is at distance at most $\ell=o(\log^5d)$ from the root, we have that $|C_x|=o(\log^6d)$. Since $d_H=(1+o(1))d$, we have that $Z_x$ stochastically dominates $Z\sim \min\left\{\log d, Bin\left((1-o(1))(d-|C|),p\right)\right\}$. 
    By Lemma \ref{l: chernoff},
    \begin{align*}
    \mathbb{P}\left(Bin(d,p)\ge \log d\right)\le \left(\frac{\alpha e}{\log d}\right)^{\log d}<\frac{1}{d^5}.
    \end{align*}
    Thus, $$\mathbb{E}[Z]\ge (1-o(1))(d-|C|)p-d\cdot \frac{1}{d^5}\ge (1-o(1))(d-|C|)p.$$
    Since $|C|\le \frac{d}{2}+\sqrt{d}$ and $\alpha>e$, we have that $\mathbb{E}[Z]>1$. Furthermore, recall that $\ell=\omega(1)$. Therefore, standard results (see, for example, \cite{LPY95,D19}) imply that with probability at least $(1-o(1))\zeta_{(d-|C|)p}$, the number of leaves after exposing $\ell-1$ layers (that is, in the layer $L_{i-1}$ if we start from $u_0$, and in the layer $L_{i+1}$ if we start from $u_1$) is at least $\left(p(d-|C|)\right)^{0.9\ell}$. Similarly, $Z_x$ is stochastically dominated by $Bin(d,p)$, and thus the number of leaves after exposing $\ell-1$ layers is at most $\left(p(d-|C|)\right)^{1.1\ell}$. 
\end{enumerate}
\end{proof}

\section{Proof of Theorem \ref{th: main}\ref{subcritical}}\label{s: subcritical}
The proof follows from a first-moment argument. Let $X_{\delta d}$ be the number of increasing paths of length $\delta d$ in $Q^d_p$. We have at most $\binom{d}{\delta d}$ ways to choose the coordinates which change along the increasing path and $(\delta d)!$ ways to order them. We then have $2^{(1-\delta)d}$ ways to choose the coordinates which are fixed along the path. This determines a path of length $\delta d$, whose edges appear in $Q^d_p$ with probability $p^{\delta d-1}$. Therefore,
\begin{align*}
    \mathbb{E}\left[X_{\delta d}\right]&=2^{(1-\delta)d}\binom{d}{\delta d}(\delta d)!p^{\delta d-1}\\
        &=2^{(1-\delta)d}\frac{d!}{\left((1-\delta)d\right)!}\left(\frac{\alpha}{d}\right)^{\delta d-1}.
\end{align*}
By Stirling's approximation and since $\alpha<e$, we have that
\begin{align*}
    \mathbb{E}\left[X_{\delta d}\right]&\le 10d\cdot 2^{(1-\delta)d}\cdot \frac{(d/e)^d}{\left((1-\delta)d/e\right)^{(1-\delta)d}}\cdot \left(\frac{\alpha}{d}\right)^{\delta d}\\
    &\le 10d\cdot \left[\frac{\alpha}{e}\cdot \left(\frac{2}{1-\delta}\right)^{(1-\delta)/\delta}\right]^{\delta d}.
\end{align*}
Since $\lim_{\delta\rightarrow 1^-}\left(\frac{2}{1-\delta}\right)^{\frac{1-\delta}{\delta}}=1$, for any constant $\alpha<e$, we can choose a constant $\delta\in [0,1)$, sufficiently close to $1$, such that $\mathbb{E}\left[X_{\delta d}\right]= o(1)$. Thus, \textbf{whp}, there is no increasing path of length at least $\delta d$.

\section{Proof of Theorem \ref{th: main}\ref{supercritical}}\label{s: supercritical}
We focus mainly on showing that $\mathbb{P}\left(\ell(Q^d_p)=d\right)=(1+o(1))\zeta_{\alpha}^2$, as the other parts follow with a slight modification of the arguments (which we will argue for at the end of the section). We begin by showing, through a careful second-moment argument, that the existence of a path of length $d$ is not very unlikely (at least inverse-polynomial in $d$ in probability).

To the task at hand, we start with the second-moment argument.
\begin{lemma}\label{l: second moment}
Let $p=\frac{\alpha}{d}$ with $\alpha\ge e$. Then $\mathbb{P}\left(\ell(Q^d_p)=d\right)\ge \frac{1}{d^5}$.
\end{lemma}
\begin{proof}
Let $X$ be the random variable counting the number of increasing paths of length $d$ in $Q^d_p$. By Paley-Zygmund,
\begin{align*}
    \mathbb{P}\left(X>0\right)\ge \frac{\mathbb{E}[X]^2}{\mathbb{E}[X^2]}.
\end{align*}
We have that
\begin{align}\label{eq: first moment}
    \mathbb{E}[X]=d!p^d.
\end{align}
Let us turn our attention to $\mathbb{E}[X^2]$. Let $\Pi$ be the set of all increasing paths of length $d$ in $Q^d$. Let $id\in \Pi$ be $id=\{v_0,\ldots, v_d\}$, where $v_i\in id$ is the vertex whose first $i$ coordinates are one, and the others are zero. Furthermore, given $\pi_1,\pi_2\in \Pi$, we stress that $\pi_1\cap \pi_2, \pi_1\cup \pi_2,$ and $\pi_1\setminus \pi_2$ are all with respect to the edges of $\pi_1$ and $\pi_2$. We then have:  
\begin{align}
    \mathbb{E}[X^2]&=\sum_{\pi_1,\pi_2\in \Pi}\mathbb{P}\left(\pi_1\in Q^d_p\land \pi_2\in Q^d_p\right)=\sum_{\pi_1,\pi_2\in \Pi}p^{|\pi_1\cup\pi_2|}\nonumber\\
    &=\sum_{\pi_1\in \Pi}p^d\cdot\sum_{\pi_2\in \Pi}p^{|\pi_2\setminus\pi_1|}=d!p^d\cdot\sum_{\pi\in \Pi}p^{|\pi\setminus id|}\nonumber\\
    &=d!p^d\sum_{\pi\in\Pi}p^{d-|\pi\cap id|}=d!p^{2d}\sum_{\pi\in \Pi}p^{-|\pi\cap id|}. \label{eq: key}
\end{align}
Now,
\begin{align}
    \sum_{\pi\in \Pi}\left(\frac{d}{\alpha}\right)^{|\pi\cap id|}=\sum_{k=0}^d\left(\frac{d}{\alpha}\right)^k\cdot Y_k, \label{eq: 1}
\end{align}
where $Y_k$ is the number of increasing paths which intersect with $id$ on exactly $k$ edges. 

Let us now estimate $Y_k$. When $k=d$, we have that $Y_k=1$. Suppose that $k<d$. Assume $\pi\in\Pi$ intersects with $id$ on exactly $k$ edges. There are $\ell\ge 1$ disjoint maximal segments of $id$ where $\pi$ does not intersect with $id$. By choosing the first and last vertex of each segment, there are at most $\binom{d}{2\ell}$ ways to choose where these segments lie. Let us denote the number of edges in each of these segments by $x_1,\ldots, x_{\ell}$, where $\sum_{i=1}^{\ell}x_i=d-k$ (note that when choosing where the segments lie, we determined $x_1,\ldots, x_{\ell})$. Since each maximal segment where two paths disagree has at least two edges, we have that $x_i\ge 2$ for $1\le i \le \ell$. In particular, $1\le \ell \le \frac{d-k}{2}$. There are at most $x_i!$ ways to form each segment. Therefore, for $k<d$,
\begin{align*}
    Y_k\le \sum_{\ell=1}^{\frac{d-k}{2}}\binom{d}{2\ell} \max_{\substack{x_1,\ldots,x_{\ell}\geq 2\\ \sum_{i=1}^{\ell}x_i=d-k}}\left\{\prod_{i=1}^\ell x_i!\right\}.
\end{align*}
Since for every $1\le i \le \ell$, we have that $x_i\ge 2$ and $\sum_{i=1}^{\ell}x_i=d-k$, we conclude that $\prod_{i=1}^{\ell}x_i!$ is maximised when every $x_i=2$ for all $i\in [\ell]$ but one of them, say $x_1$, which is equal to $d-k-2(\ell-1)$. 
Hence, for $k<d$,
\begin{align*}
    Y_k\le \sum_{\ell=1}^{\frac{d-k}{2}}\binom{d}{2\ell}2^{\ell-1}(d-k-2\ell+2)!\,.
\end{align*}

Returning to \eqref{eq: 1}, we have that
\begin{align}
    \sum_{k=0}^d\left(\frac{d}{\alpha}\right)^k\cdot Y_k &\le d!+\sum_{\ell=1}^{\frac{d}{2}}2^{\ell}\binom{d}{2\ell}\sum_{k=0}^{d-2\ell}\left(\frac{d}{\alpha}\right)^k(d-k-2\ell+2)!\,, \label{eq: 2}
\end{align}
where the first summand, $d!$, corresponds to the case where $k=d$ --- indeed, we use the fact that $\left(\frac{d}{\alpha}\right)^d\le \left(\frac{d}{e}\right)^d\le d!$ and that $Y_k=1$. Now,
\begin{align*}
   \sum_{k=0}^{d-2\ell}\left(\frac{d}{\alpha}\right)^k(d-k-2\ell+2)!&=\left(\frac{d}{\alpha}\right)^{d-2\ell}\sum_{m=0}^{d-2\ell}\left(\frac{\alpha}{d}\right)^m(m+2)!\\
   &\le \left(\frac{d}{\alpha}\right)^{d-2\ell}d^2\sum_{m=0}^{d-2\ell}\left(\frac{\alpha}{d}\right)^mm!\\
   &\le \left(\frac{d}{\alpha}\right)^{d-2\ell}d^3\sum_{m=0}^{d-2\ell}\left(\frac{\alpha}{d}\right)^m\left(\frac{m}{e}\right)^m\\
   &\le \left(\frac{d}{\alpha}\right)^{d-2\ell}d^4 \cdot \left(\frac{\alpha}{e}\right)^{d-2\ell} \\
   &\le \left(\frac{d}{e}\right)^{d-2\ell}d^4,
\end{align*}
where in the first inequality we used the fact that $\ell\ge 1$ and in the penultimate inequality we used our assumption $\alpha\ge e$ and the fact that $d\le m$.
Returning to \eqref{eq: 2}, we now have
\begin{align*}
    \sum_{k=0}^d\left(\frac{d}{\alpha}\right)^k\cdot Y_k&\le d!+\sum_{\ell=1}^{\frac{d}{2}}2^{\ell}\binom{d}{2\ell}\left(\frac{d}{e}\right)^{d-2\ell}d^4\\
   &\le d!+\left(\frac{d}{e}\right)^dd^4\sum_{\ell=1}^{\frac{d}{2}}\left(\frac{ed}{\ell}\cdot\frac{e}{d}\right)^{2\ell}\\
   &\le d!\cdot d^5.
\end{align*}
Substituting the above in \eqref{eq: key}, we obtain that
\begin{align*}
    \mathbb{E}[X^2]&\le \left(p^{d}d!\right)^2d^5.
\end{align*}
By Paley-Zygmund, the above, and \eqref{eq: first moment}, 
\begin{align*}
    \mathbb{P}\left(X>0\right)\ge \frac{1}{d^5}.
\end{align*}
\end{proof}
We note that in the proof above, we did not attempt to optimise the exponent in $d^{-5}$, and instead aimed for simplicity and clarity --- indeed, with more careful calculations one could obtain a better exponent, however, that does not affect the rest of the proof.

We now turn our attention to showing that if $\textbf{0}$ and $\textbf{1}$ both `percolate', then \textbf{whp} we can find $d^{10}$ pairwise disjoint subcubes, $Q(v_{1,0}, v_{1,1}),\ldots, Q(v_{d^{10},0},v_{d^{10},1})$, of dimension $(1-o_d(1))d$, such that there is a path in $Q^d_p$ between $\textbf{0}$ and every $v_{i,0}$, and between $\textbf{1}$ and every $v_{i,1}$. Showing the existence of such subcubes requires a delicate construction. Throughout the rest of this section, we assume that $p=\frac{\alpha}{d}$ for some $\alpha>e$, and recall that $\zeta_{\alpha}$ is defined according to \eqref{survival prob}.

We begin by showing that with probability at least $(1-o(1))\zeta_{\alpha}^2$, both $\textbf{0}$ and $\textbf{1}$ `percolate' in $Q^d_p$ and we can find a tree rooted at $\textbf{1}$ which is `above' the tree rooted at $\textbf{0}$. Formally,
\begin{lemma}\label{l: first step}
With probability at least $(1-o(1))\zeta_{\alpha}^2$, the following holds. There exist two trees $T_0', T_1'\in Q^d_p$ such that $T_0'$ has all its leaves at $L_{\log\log d}$, $T_1'$ has all its leaves at $L_{d-\log\log d}$, and:
\begin{enumerate}
    \item $\textbf{0}\in V(T_0'), \textbf{1}\in V(T_1')$; and,
    \item $|L_{\log\log d}\cap V(T_0')|\ge \log d$, $|L_{d-\log\log d}\cap V(T_1')|\ge \log d$; and,
    \item $|I(L_{\log\log d}\cap V(T_0'))|\le \log^{2\alpha}d$, $|[d]\setminus I(L_{d-\log\log d}\cap V(T_1'))|\le \log^{2\alpha}d$; and, 
    \item for every two leaves of $T_0'$, $w_{1,0},w_{2,0}\in V(T_0')$ with $w_{1,0}\neq w_{2,0}$, we have that $I(w_2)\cap C_{T_0'}(w_1,0)\neq\varnothing$; and,
    \item for every two leaves of $T_1'$, $w_{1,1}, w_{2,1}$ with $w_{1,1}\neq w_{2,1}$, we have that $([d]\setminus I(w_2))\cap C_{T_1'}(w_1,1)\neq\varnothing$; and,
    \item for every $v_0\in V(T_0')$, we have that $|C_{T_0'}(v_0,0)|\le \log^2d$, and for every $v_1\in V(T_1')$ we have that $|C_{T_1'}(v_1,1)|\le \log^2d$; and,
    \item for every $v_0\in V(T_0')$ and $v_1\in V(T_1')$, $v_0<v_1$.
\end{enumerate}
\end{lemma}
\begin{proof}
Run the Tree Construction algorithm described in Section \ref{s: bfs} with the following inputs. Let \ref{input: subcube}, the subcube, be $Q^d$, let \ref{input: root}, the root, be $\textbf{0}$, let \ref{input: coordinates}, the set of coordinates which we avoid, be $\varnothing$, let \ref{input: layer}, the layer at which we truncate, be $L_{\log \log d+1}$, and let \ref{input: probability} be a sequence of Bernoulli$(p)$ random variables. Let $T_0'$ be the tree this algorithm outputs. 

Item $4$ follows deterministically from Proposition \ref{proposition: tree construction}\ref{property: leaves}, and item $6$, with respect to $T_0'$, follows deterministically from Proposition \ref{proposition: tree construction}\ref{property: list}. Note the dimension of the subcube in the algorithm's input is $d$, that we do not avoid any coordinates, and that the distance of the layer we truncate at from the root is $\log\log d+1$. Hence, we may apply Proposition \ref{proposition: tree construction}\ref{property: queries}. As $pd=\alpha$, items $1$, $2$, and $3$, with respect to $T_0'$, follow from Proposition \ref{proposition: tree construction}\ref{property: queries}. Let us denote the event that such $T_0'$ exists by $\mathcal{A}_0$, where we have that $\mathbb{P}(\mathcal{A}_0)\ge (1-o(1))\zeta_{\alpha}$ by \ref{proposition: tree construction}\ref{property: queries}.

Let $I_0\coloneqq I\left(V(T_0)\cap L_{\log\log d}\right)$. Conditioned on $\mathcal{A}_0$, we have that $|I_0|\le \log^{2\alpha}d$. We now run the Tree Construction algorithm described in Section \ref{s: bfs} with the following inputs. Let \ref{input: subcube}, the subcube, be $Q^d$, let \ref{input: root}, the root, be $\textbf{1}$, let \ref{input: coordinates}, the set of coordinates which we avoid, be $I_0$, let \ref{input: layer}, the layer at which we truncate, be $L_{d-\log \log d-1}$, and let \ref{input: probability} be a sequence of Bernoulli$(p)$ random variables. Let $T_1'$ be the tree this Tree Construction algorithm outputs. 

Note that item $7$ holds by construction. Item $5$ follows deterministically from Proposition \ref{proposition: tree construction}\ref{property: leaves}, and item $6$, with respect to $T_1'$, follows deterministically from Proposition \ref{proposition: tree construction}\ref{property: list}. Note that the dimension of the subcube in the algorithm's input is $d$, that $|I_0|\le \frac{d}{2}+\sqrt{d}$, and that the distance of the layer we truncate at from the root is $\log\log d+1$. Thus, we may apply Proposition \ref{proposition: tree construction}\ref{property: queries}. As $(d-|I_0|)p=(1+o(1))\alpha$, items $1$, $2$, and $3$, with respect to $T_1'$, follow from Proposition \ref{proposition: tree construction}\ref{property: queries}. Let us denote the event that such $T_1'$ exists by $\mathcal{A}_1$, and we have that, conditional on $\mathcal{A}_0$, the probability that $\mathcal{A}_1$ holds is at least $(1-o(1))\zeta_{(1-o(1))\alpha}=(1-o(1))\zeta_{\alpha}$ by \ref{proposition: tree construction}\ref{property: queries}.

Furthermore, the sets of edges explored during the construction of $T_0'$ and that during the construction of $T_1'$ are disjoint, and therefore the probability of both events $\mathcal{A}_0$ and $\mathcal{A}_1$ holding is at least $(1-o(1))\zeta_{\alpha}^2$.
\end{proof}

We now turn to show that \textbf{whp} we can extend $T_0'$ to a tree of height $O(\log d)$, with, say, $d^{11}$ leaves, such that the leaves are easily distinguishable. That is, our goal is that every leaf of $T_0'$ will be uniquely identified with a set of coordinates, whose order is, say, $150\log^{2}d$. We also aim to extend $T_1'$, maintaining all of its vertices above (the extension of) $T_0'$. We do so, roughly, by growing $T_0'$ on the first $d/2$ coordinates, and $T_1'$ on the last $d/2$ coordinates, utilising Proposition \ref{proposition: tree construction}\ref{property: queries}. More precisely,
\begin{lemma}\label{l: second step}
Suppose $T_0'$ and $T_1'$ satisfy the properties in the statement of Lemma \ref{l: first step}. Then, \textbf{whp}, there exist trees $T_0, T_1\in Q^d_p$ such that $T_0$ has all its leaves at $L_{150\log d}$, $T_1$ has all its leaves at $L_{d-150\log d}$, and:
\begin{enumerate}
    \item $\textbf{0}\in V(T_0)$ and $\textbf{1}\in V(T_1)$; and,
    \item $|L_{150\log d}\cap V(T_0)|\ge d^{11}$ and $|L_{d-150\log d}\cap V(T_1)|\ge d^{11}$; and,
    \item for every $v\in L_{150\log d}\cap V(T_0)$, we have that $|C_{T_0}(v,0)|\le 150\log^{2}d$; and,
    \item for every $v\in L_{d-150\log d}\cap V(T_1)$, we have that $|C_{T_1}(v,1)|\le 150\log^2d$; and,
    \item for every $u,v\in L_{150\log d}\cap V(T_0)$ with $u\neq v$, we have that $I(u)\cap C_{T_0}(v,0)\neq \varnothing$; and,
    \item for every $u,v\in L_{d-150\log d}\cap V(T_1)$ with $u\neq v$ we have that $([d]\setminus I(u))\cap C_{T_1}(v,1)\neq \varnothing$; and,
    \item for every $v_0\in V(T_0)$ and $v_1\in V(T_1)$, $v_0<v_1$.
\end{enumerate}
\end{lemma}
\begin{proof}
Note that when constructing $T_0'$ and $T_1'$, we only considered edges up to the $\log\log d$-th, and $d-\log\log d$-th layer (respectively). We may thus assume that edges crossing the other layers have not been exposed yet.

We begin by showing that \textbf{whp} $T_0$ exists. Let $I_1\coloneqq I\left(V(T_1')\cap L_{d-\log\log d}\right)$, where we note that by assumption $|I_1|\ge d-\log^{2\alpha}d$. Furthermore, since $T_0'$ satisfies the properties of Lemma \ref{l: first step}, we have that $\textbf{0}\in V(T_0')$ and $|L_{\log\log d}\cap V(T_0')|\ge \log d$. Let $U_0=\{u_1,\ldots, u_{\log d}\}$ be a set of arbitrary $\log d$ vertices from $L_{\log\log d}\cap V(T_0')$. For every $i\in [\log d]$, let $w_i$ be defined as $I(w_i)=[d]\setminus C_{T_0'}(u_i,0)$. Since $I(u_i)\cap C_{T_0'}(u_i,0)=\varnothing$, we have that $w_i>u_i$. By our assumption, for every two leaves $u_i\neq u_j\in V(T_0')\cap L_{\log\log d}$, we have that $I(u_i)\cap C_{T_0'}(u_j,0)\neq \varnothing$. Furthermore, by our assumption, for every $i\in[\log d]$, we have that $|C_{T_0'}(u_i,0)|\le \log^2d$. Thus, $Q(1)\coloneqq Q(u_1, w_1),\ldots, Q(\log d)\coloneqq Q(u_{\log d}, w_{\log d})$ form a set of $\log d$ pairwise disjoint subcubes of dimension at least $d-\log d-\log^2d\ge d-2\log^2d$. We claim that with probability bounded away from zero, there exists in $Q(i)_p$ a tree $B_i$, such that $B_i$ together with the path from $\textbf{0}$ to $u_i$ in $Q^d_p$ form a suitable choice for $T_0$. 

To that end, for every $i\in [\log d]$, we run the Tree Construction algorithm given in Section \ref{s: bfs} with the following inputs. Let \ref{input: subcube}, the subcube, be $Q(i)$, let \ref{input: root}, the root, be $u_i$, let \ref{input: coordinates}, the set of coordinates which we avoid, be $\left([d]\setminus I_1\right)\cup \left([d]\setminus [d/2]\right)$, let \ref{input: layer},  the layer at which we truncate, be $L_{150\log d}$, and let \ref{input: probability}, be a sequence of Bernoulli$(p)$ random variables. Let $B_i$ be the tree this algorithm outputs. 

Note that the subcube in the algorithm's input is of dimension $(1-o(1))d$, that the set of coordinates we avoid, $\left([d]\setminus I_1\right)\cup \left([d]\setminus [d/2]\right)$, is of size at most $\log^{2\alpha}d+\frac{d}{2}\le \frac{d}{2}+\sqrt{d}$, and that the distance of the layer we truncate at from the root is $150\log d-\log\log d$. Thus, we may apply Proposition~\ref{proposition: tree construction}\ref{property: queries}, and obtain that with probability at least $c$, for some constant $c>0$, $|L_{150\log d}\cap V(B_i)|\ge \left(\frac{2e}{5}\right)^{135\log d}\ge d^{11}$.

In the event $|L_{150\log d}\cap V(B_i)|\ge d^{11}$, note that letting $T_0$ be $B_i$ with the path from $\textbf{0}$ to $u_i$ in $Q^d_p$, for every $v\in V(B_i)$, we have that $C_{T_0}(v,0)=C_{B_i}(v,0)$, and thus by Proposition \ref{proposition: tree construction}\ref{property: list}, $|C_{T_0}(v,0)|\le 150\log^2d$. Since we restrict the queries to $\left([d/2]\cap I_1\right)$, for every $v_0\in V(T_0)$, we have that $I(v_0)\subseteq I_1$ and $I(v_0)\subseteq [d/2]\cup I(V(T_0'))$, where the first further implies that for every $v_1\in V(T_1')$, we have $v_0<v_1$. Finally, we need to verify that for every $u,v\in L_{150\log d}\cap V(T_0)$ with $u\neq v$, we have that $I(u)\cap C_{T_0}(v,0)\neq \varnothing$ --- indeed, this follows from Proposition \ref{proposition: tree construction}\ref{property: leaves}. Therefore, with probability at least $c>0$, the path from $\textbf{0}$ to $u_i$ in $Q^d_p$ appended with $B_i$ forms a suitable choice for $T_0$.

Since these events are independent for every $i\in [\log d]$, the probability there is no such $T_0$ is at most $(1-c)^{\log d}=o_d(1)$.

The existence of $T_1$ follows similarly, where we note that $I(V(T_0))\subseteq I(V(T_0'))\cup [d/2]$ and $|I(V(T_0'))|\le \log^{2\alpha}d=o(d)$.
\end{proof}

We are now ready to show that \textbf{whp} we can extend $T_1$, so that we can find $d^{10}$ pairwise disjoint subcubes, of dimension $(1-o(1))d$, with their antipodal points connected to $\textbf{0}$ and $\textbf{1}$.
\begin{lemma}\label{l: extending t1 second}
Suppose $T_0$ and $T_1$ satisfy the properties of Lemma \ref{l: second step}. Then, \textbf{whp}, we can find subsets $V_0=\{v_{1,0},\ldots, v_{d^{10},0}\}\subseteq V(T_0)\cap L_{150\log d}$ and $V_1=\{v_{1,1},\ldots, v_{d^{10},1}\}$, with $V_1$ being at layer at least $d-15^3\log^{6}d$, such that the following holds.
\begin{enumerate}
    \item For every $v_0\in V_0$ and $v_1\in V_1$, there is a path in $Q^d_p$ between $\textbf{0}$ and $v_0$ and between $\textbf{1}$ and $v_1$.
    \item for every $i\in [d^{10}]$, we have that $v_{i,0}<v_{i,1}$; and,
    \item for every $i\in [d^{10}]$, we have that $C_{T_0}(v_{i,0},0)\cap I(v_{i,1})=\varnothing$.
\end{enumerate}
\end{lemma}
\begin{proof}
Once again, we stress that we have not queried any of the edges between the $150\log d$-th and the $d-150\log d$-th layers thus far.

Let $V_0=\{v_{1,0},\ldots, v_{d^{10},0}\}$ be an arbitrary set of $d^{10}$ vertices in $V(T_0)\cap L_{150\log d}$. Note that $|I(v_{i,0})|=150\log d$ for every $i\in[d^{10}]$.

Let $M=\{m_1, \ldots, m_{d^{10}\cdot d}\}$ be an arbitrary set of $d^{11}$ vertices in $L_{d-150\log d}\cap V(T_1)$. We arbitrarily split them to sets $S_1, \ldots, S_{d^{10}}$, each of order $d$. For every $i\in [d^{10}]$, we associate $S_i$ with $v_{i,0}$. We now turn to show that \textbf{whp}, for every $i\in [d^{10}]$ there exists a tree $T(S_i)$ in $Q^d_p$, rooted at some vertex in $S_i$, such that at least one of its vertices $y\in V\left(T(S_i)\right)\cap\bigcup_{s=d-15^3\log^6d}^{d}L_s$, satisfies that $v_{i,0}<y$ and $I(y)\cap C_{T_0}(v_{i,0})=\varnothing$. 

Fix $i\in [d^{10}]$. Denote the vertices of $S_i$ by $\{u_1, \ldots, u_{d}\}$. For every $k\in [d]$, let $w_k$ be the vertex in $Q^d$ defined by $I(w_k)=C_{T_1}(u_k,1)$. Similar to the proof of Lemma \ref{l: second step}, we have that $Q(1)\coloneqq Q(w_1,u_1),\ldots, Q(d)\coloneqq Q(w_d, u_d)$ form $d$ pairwise vertex disjoint subcubes. By our assumption, $|C_{T_1'}(u_k,1)|\le 150\log^2d$, and thus each of these subcubes is of dimension at least $d-200\log^2d$. Fix $k\in [d]$ and consider $Q(k)$. Run the Tree Construction algorithm given in Section \ref{s: bfs} with the following inputs. Let \ref{input: subcube}, the subcube, be $Q(k)$, let \ref{input: root}, the root, be $u_k$, let \ref{input: coordinates}, the set of coordinates which we avoid, be $I(v_{i,0})$, let \ref{input: layer},  the layer at which we truncate, be $L_{d-10\log^4 d-1}$, and let \ref{input: probability} be a sequence of Bernoulli$(p)$ random variables. Let $B_k$ be the tree this algorithm outputs. Similarly to previous arguments, by Proposition \ref{proposition: tree construction}\ref{property: queries}, since $(d-20\log^4d)p\ge \frac{\alpha}{2}>\frac{13}{10}$, with probability at least $c$, for some positive constant $c$, we have $|V_{d-\log^4 d}\cap V(B_k)|\ge \left(\frac{13}{10}\right)^{9\log^4d}$. Note that every vertex in $B_k$ is above $v_{i,0}$. Thus, with probability at least $1-(1-c)^{d}=1-o\left(d^{-11}\right)$, there exists a set of $\left(\frac{13}{10}\right)^{9\log^4d}$ vertices in layer $L_{d-10\log^4d}$, all of which are connected to $\textbf{1}$ in $Q^d_p$ and all of which are above $v_{i,0}$. Denote this set of vertices by $X$.

Initialise $X(0)\coloneqq X$ and $C(0)\coloneqq C_{T_0}(v_{i,0})$. Now, at each iteration $j\in [|C(0)|]$, we proceed as follows. Let $\ell$ be the first (smallest) coordinate in $C(j-1)$. If at least $\frac{|X(j-1)|}{d}$ vertices $x\in X(j-1)$ have that $\ell \in [d]\setminus I(x)$, we set $X(j)$ to be these set of vertices and update $C(j)=C(j-1)\setminus \{\ell\}$. Otherwise, there are at least $\left(1-\frac{1}{d}\right)|X(j-1)|$ vertices $x\in X(j-1)$, such that $\ell\in I(x)$. Each vertex $x$ has at least one neighbour $x'$ in the layer below, such that $\ell\in [d]\setminus I(x')$. The probability that $xx'$ is in $Q^d_p$ is $p$, and these are independent trials for every vertex $x$. Thus, the number of $x'$ such that $\ell\in[d]\setminus I(x')$ and that there is $x\in X(j-1)$ such that $xx'\in E(Q^d_p)$ stochastically dominates $Bin\left(\left(1-\frac{1}{d}\right)|X(j-1)|,p\right)$. By Lemma \ref{l: chernoff}, with probability at least $1-\exp\left(-\frac{|X(j-1)|}{10d}\right)$, we have that the number of such $x'$ is at least $\frac{|X(j-1)|}{2d}$, where we stress that if $|X(j-1)|\ge d^2$, this holds with probability at least $1-o(d^{-11})$. We then let $X(j)$ be the set of such $x'$, and update $C(j)=C(j-1)\setminus\{\ell\}$. Repeating the above process for $|C(0)|\le 150\log^2d$ iterations, by the union bound we have that with probability at least $1-o(\log^2d\cdot d^{-11})$ at each iteration, $|X(j)|\ge \frac{|X(j-1)|}{2d}$, and in particular, $$|X(|C(0)|)|\ge \frac{\left(\frac{13}{10}\right)^{9\log^4d}}{(2d)^{\log^2d}}\ge d^2.$$
Thus, with probability at least $1-o(d^{-10})$ there exists a proper choice for $v_{i,1}$ at layer at least $d-10\log^4d\cdot150\log^2d\ge d-15^3\log^6d$. Union bound over the $d^{10}$ choices of $i$ completes the proof.
\end{proof}

We are now ready to prove Theorem \ref{th: main}\ref{supercritical}.
\begin{proof}[Proof of Theorem \ref{th: main}\ref{supercritical}]
By Lemmas \ref{l: first step} through \ref{l: extending t1 second}, with probability at least $(1-o(1))\zeta_{\alpha}^2-o(1)=(1-o(1))\zeta_{\alpha}^2$, we can find subsets $V_0=\{v_{1,0},\ldots, v_{d^{10},0}\}\subseteq V(T_0')\cap L_{150\log d}$ and $V_1=\{v_{1,1},\ldots, v_{d^{10},1}\}$, with $V_1$ being at layer at least $d-15^3\log^{6}d$, such that the following holds.
\begin{enumerate}
    \item For every $v_0\in V_0$ and $v_1\in V_1$, we have that there is a path in $Q^d_p$ between $\textbf{0}$ and $v_0$ and between $\textbf{1}$ and $v_1$.
    \item for every $i\in [d^{10}]$, we have that $v_{i,0}<v_{i,1}$; and,
    \item for every $i,j\in [d^{10}]$ with $i\neq j$, we have that $C_{T_0}(v_{i,0},0)\cap I(v_{i,1})=\varnothing$.
\end{enumerate}
By our construction of $T_0$, for every $u,v\in V(T_0)$ we have that $I(u)\cap C_{T_0}(v,0)\neq \varnothing$. This implies that $I(u)\cap \left([d]\setminus I(v)\right)\neq \varnothing$. Therefore, for every $i,j\in [d^{10}]$ with $i\neq j$, we have that the subcubes $Q(v_{i,0},v_{i,1})$ and $Q(v_{j,0},v_{j,1})$ are disjoint. 

For every $i\in[d^{10}]$, let $H(i)=Q(v_{i,0},v_{i,1})$. By Lemma \ref{l: second moment}, with probability at least $\frac{1}{d^5}$, there is a path in $H(i)_p$ between $v_{i,0}$ and $v_{i,1}$, Thus, the probability that there is no path between $\textbf{0}$ and $\textbf{1}$ in $Q^d_p$ is at most $$1-(1-o(1))\zeta_{\alpha}^2+\left(1-\frac{1}{d^5}\right)^{d^{10}}=1-(1-o(1))\zeta_{\alpha}^2.$$
Standard results (see, for example, \cite{D19}) implies that there does not exist a path in $Q^d_p$ from $\textbf{0}$ to any vertex in $L_{\log d}$ with probability at least $(1+o(1))(1-\zeta_{\alpha})$. Similarly, the probability there is no path from $\textbf{1}$ to any vertex in $L_{d-\log d}$ is at least $(1+o(1))(1-\zeta_{\alpha})$. Noting that these two events are independent, and since the function $f(x)=x^2+2x(1-x)$ is increasing in the interval $(0,1)$, the probability that at least one such path does not exist is at least 
$$(1-\zeta_{\alpha})^2+2\zeta_{\alpha}(1-\zeta_{\alpha})-o(1)=(1+o(1))\left(1-\zeta_{\alpha}^2\right).$$
Therefore, $\mathbb{P}\left(\ell(Q^d_p)=d\right)=(1+o(1))\zeta_{\alpha}^2$. 

The statements for $d-1, d-2$ follow similarly. For $d-1$, in Lemma \ref{l: first step} we can grow $T_0'$ (or $T_1'$) from $L_1$ (or, respectively, $L_{d-1}$), where in a manner similar to the above Lemmas we can show that \textbf{whp} at least one of the vertices from these layers succeeds. As for $d-2$, we can do so from both layers and thus the statement holds \textbf{whp}. 
\end{proof}

\begin{remark}\label{r: distribution}
Given any pair of antipodal vertices $(u_0, u_1)$, one can define a $u_0-u_1$ increasing path: we say that $P=\{v_1, \ldots, v_i, v_{i+1}, \ldots, v_d\}$, where $v_1=u_0$ and $v_d=u_1$, is $u_0-u_1$ increasing if the for every $i\in [d]$, $v_{i+1}$ is closer (in Hamming distance) to $u_1$ than $v_i$ is. The above proof suggests that given two pairs of antipodal vertices, $(u_0,u_1)$ and $(v_0,v_1)$, we can construct a $u_0-u_1$ increasing path and a $v_0-v_1$ increasing path independently with probability $(1+o(1))\zeta_{\alpha}^4$ if the Hamming distance of $u_0$ and $v_0$ is at least $\log^7d$. Given a vertex $v$, the number of vertices at distance at most $\log^7d$ from $v$ is $O\left(\binom{d}{\log^7d}\right)$. As $O\left(\binom{d}{\log^7d}\right)=o(2^{d})$, it follows from Chebyshev's inequality that \textbf{whp} the number of increasing antipodal paths in $Q^d_p$ is $(1+o(1))\zeta_{\alpha}^22^{d-1}$. Furthermore, let $Z$ be the number of increasing antipodal paths in $Q^d_p$ and let $\overline{Z}=(Z-\zeta_{\alpha}^22^{d-1})/\sqrt{\zeta_{\alpha}^2(1-\zeta_{\alpha}^2)2^{d-1}}$. By calculating higher moments of $\overline{Z}$ one can show that $\overline{Z}$ converges in distribution to the standard normal distribution (see also \cite[Theorem 1.4 part a-ii]{R13}). 
\end{remark}

\paragraph{Acknowledgement} The authors wish to thank Ross Pinsky for his comments on an earlier version of the paper, and for bringing reference \cite{R13} to our attention. 
\bibliographystyle{abbrv}
\bibliography{perc} 
\end{document}